\documentclass[12pt, leqno]{article}
\usepackage{amsmath, amsfonts, amsthm}
\newtheorem{thm}{Theorem}
\newtheorem{lem}[thm]{Lemma}

\begin{document}
\author{Ajai Choudhry}
\title{An improvement of Prouhet's 1851 result\\
on multigrade chains}
\date{}
\maketitle

\begin{abstract}
In 1851 Prouhet showed that when $N=j^{k+1}$ where $j$ and $k$ are positive integers, $j \geq 2$,  the first $N$ consecutive positive integers can be separated into $j$ sets, each set containing $j^k$ integers, such that the sum of the $r$-th powers of the members of each  set is the same for $r=1,\,2,\,\ldots,\,k$. In this paper we show that even when $N$ has the much smaller value $2j^k$, the first $N$ consecutive positive integers can be separated into $j$ sets, each set containing $2j^{k-1}$ integers, such that the integers of each set  have equal sums of $r$-th powers  for $r=1,\,2,\,\ldots,\,k$. Moreover, we show that this can be done in at least $\{(j-1)!\}^{k-1}$ ways. We also show that there are infinitely many other positive integers $N=js$ such that the first $N$ consecutive positive integers can   similarly be separated into $j$ sets of integers, each set containing $s$ integers,  with equal sums of  $r$-th powers  for $r=1,\,2,\,\ldots,\,k$, with the value of $k$ depending on the integer $N$.
\end{abstract}



\section{Introduction}
If there exist integers $a_{uv},\;u=1,\,2,\,\ldots,\,s,\;v=1,\,2,\,\ldots,\,j$ ($j$ and $s$ being positive integers $ \geq 2$), such that the relations
\begin{equation}
\sum_{u=1}^sa_{u1}^r=\sum_{u=1}^sa_{u2}^r=\cdots =\sum_{u=1}^sa_{uj}^r, \label{basicchn}
\end{equation}
are satisfied when $r=1,\,2,\,\dots,\,k$, we write,
\begin{equation}
a_{11},\,a_{21},\ldots,\,a_{s1} \stackrel{k}{=} a_{12},\,a_{22},\ldots,\,a_{s2} \\
\stackrel{k}{=}  \ldots \stackrel{k}{=} a_{1j},\,a_{2j},\ldots,\,a_{sj}. \label{basicchnnot1}
\end{equation}
 A solution of \eqref{basicchn} is said to be nontrivial if the $j$  sets $\{a_{uv},\,u=1,\,2,\,\ldots,\,s\}$, $v=1,\,2,\,\ldots,\,j$, are distinct.  The least value of $s$ for which there exists a nontrivial solution of \eqref{basicchn} is denoted by $P(k,\,j)$. Relations of type \eqref{basicchn} are known as multigrade chains.

The first example of multigrade chains  was obtained  in 1851 by  Prouhet \cite[p. 449]{HW} who gave a rule to  separate the first $j^{k+1}$ positive integers into $j$ sets  that provide a multigrade chain \eqref{basicchnnot1} with $s=j^k$. Relevant excerpts from Prouhet's original note are  given in \cite[pp. 999-1000]{BP}. As a numerical example, Prouhet noted that the integers $1,\,2,\,\ldots,\,27$ can be separated into three sets  satisfying the relations, 
\begin{equation}
\begin{aligned}
1,\,6,\,8,\,12,\,14,\,16,\,20,\,22,\,27 &\stackrel{2}{=}2,\,4,\,9,\,10,\,15,\,17,\,21,\,23,\,25\\
&\stackrel{2}{=}3,\,5,\,7,\,11,\,13,\,18,\,19,\,24,\,26.
\end{aligned}
\label{chn27ex1}
\end{equation}
While Prouhet himself did not give a proof, his result has subsequently been proved by several authors in various ways \cite{Le, Ng, Ro1, Wr2, Wr3}. 

It has been proved by Wright \cite{Wr1} that $P(k,\,j) \leq (k^2+k+2)/2$ when $k$ is even and $P(k,\,j) \leq (k^2+3)/2$ when $k$ is odd. However, Wright's method proves only the existence of solutions of \eqref{basicchn} and cannot be used to construct actual examples of multigrade chains. When $j=2$, it has been shown that $P(k,\,2)=k+1$ when $k \leq 9$ \cite[p. 440, \;p. 449]{HW} and also when $k=11$ \cite{CW}. Further, it has been shown that $P(k,\,j)=k+1$ for  $k=2,\,3$ and $5$ and for all values of $j$ \cite[p. 437]{HW}.

Numerous papers have been published on Prouhet's problem, especially concerning the particular case of equations \eqref{basicchn} when $j=2$ and this problem is now referred to as the Prouhet-Tarry-Escott problem.  Gloden has written an entire book on  multigrade equations and multigrade chains \cite{Gl} and the problem has been the subject of two survey articles \cite{Bo, RN} both of which contain extensive bibliographies. Further, Prouhet's problem has been linked to  various other problems \cite{AMZ, BP, BOR, Ce, GGG}. However, despite the passage of time since the publication of Prouhet's note in 1851  and the attention  bestowed on the problem,  until now Prouhet's  original result has not been improved. 

A remarkable feature of Prouhet's solution of the equations \eqref{basicchn} is that the integers $a_{uv},\;u=1,\,2,\,\ldots,\,s,\;v=1,\,2,\,\ldots,\,j$, are a permutation of the first $N$ consecutive positive integers where $N=j^{k+1}$. The problem of separating  $N$ consecutive integers into sets with equal power sums has been considered  in two articles \cite{Ro2, Ro3} by Roberts who has shown that ``if $q$ is a factorization of $n$ whose factors have least common multiple $L_q$ then the first $n$ nonnegative integers can be split into $L_q$ classes  with equal $t$-th power sums   for all $t$ satisfying 
\[ 
0 \leq t < q^*-\max_{0 < s < L_q} \nu_s,
\]
where $q^*$ is the number of factors in $q$ and $\nu_s$ is the number of them that divide $s$". 
The maximum possible value of $t$ is relatively small and  is the smallest exponent in the canonical prime factorization of $n$.

In this paper we will show that the  consecutive positive integers $1,\,2,\,\,\ldots,$ $2j^k$ can be separated into $j$ sets of $2j^{k-1}$ members satisfying the relations \eqref{basicchnnot1}. In fact, we show that this can, in general, be done in at least $\{(j-1)!\}^{k-1}$ ways. For $j > 2$, the integer $2j^{k}$ is much smaller than $j^{k+1}$ and the result is thus a significant improvement over Prouhet's solution of \eqref{basicchnnot1}. 

 We also show that there  exist infinitely many other positive integers $N=js$ such that the positive integers $1,\,\,2,\,\ldots,\,N$ can be separated into $j$ sets, each set containing $s$ integers, such that the $j$ sets provide a solution of \eqref{basicchnnot1} and, in general, this can be done in several ways. The theorems in this paper  give much better results as compared to the results obtained by Roberts \cite{Ro2, Ro3}.

\section{Some preliminary lemmas}
\begin{lem} If there exist integers $a_{uv},\;u=1,\,2,\,\ldots,\,s,\;v=1,\,2,\,\ldots,\,j$ such that
\begin{equation}
a_{11},\,a_{21},\ldots,\,a_{s1} \stackrel{k}{=} a_{12},\,a_{22},\ldots,\,a_{s2} \\
\stackrel{k}{=}  \ldots \stackrel{k}{=} a_{1j},\,a_{2j},\ldots,\,a_{sj},
\end{equation}
then
\begin{equation}
\begin{aligned}
&Ma_{11}+K,\,Ma_{21}+K,\ldots,\,Ma_{s1}+K \\
& \quad \stackrel{k}{=} Ma_{12}+K,\,Ma_{22}+K,\ldots,\,Ma_{s2}+K \\
&\quad\stackrel{k}{=}  \ldots \\
& \quad  \stackrel{k}{=} Ma_{1j}+K,\,Ma_{2j}+K,\ldots,\,Ma_{sj}+K,
\end{aligned}
\end{equation}
where $M$ and $K$ are arbitrary integers.
\end{lem}
\begin{proof}
 When $j=2$, this is a simple consequence of the binomial theorem and is a well-known lemma \cite{Do}. When $j > 2$, then also, the lemma follows immediately from the binomial theorem.
\end{proof}

\begin{lem}For any arbitrary positive integer $j> 1 $, the first $2j$ consecutive positive integers can be separated into $j$ sets, each set containing  two integers,  such that the sum of the integers in each set is the same.
\end{lem}
\begin{proof}  The $j$ sets $\{u,\,2j+1-u\},\;u=1,\,2,\,\ldots,\,j$, have the same sum $2j+1$. Since the integers in these $j$ sets are the  first $2j$ consecutive positive integers, the lemma is proved.
\end{proof}

\begin{lem} For any arbitrary positive integers $m$ and  $j> 1 $, the first $2mj$ consecutive positive integers can be separated into $j$ sets, each set containing  $2m$ integers,  such that the sum of the integers in each set is the same.
\end{lem}
\begin{proof} This is a straightforward generalisation of Lemma 2. We first divide the  consecutive integers $1,\,2,\,\ldots,\,2mj$ into $2j$ blocks, each block   consisting  of $m$ consecutive integers -- the first block being $1,\,2,\,\ldots,\,m$. Next for each integer $u, \; 1 \leq u \leq j$, we construct a set consisting of the $m$ integers of the $u^{\rm th}$ block and the $m$ integers of the $(2j+1-u)^{\rm th}$ block. We thus get $j$ sets, each set consisting of $2m$ integers,  such that the sum of the integers in each set is $m(2mj+1)$. This proves the lemma.  
\end{proof}

\begin{lem} For any arbitrary positive integer $j> 1 $, the first $j^2$ consecutive positive integers can be separated into $j$ sets, each set containing  $j$ integers,  such that the sum of the integers in each set is the same.
\end{lem}
\begin{proof} If we separate the first $j^2$ consecutive positive integers into the  $j$ sets,
\begin{align*}
&\{1,&&j+2,&& 2j+3,&& 3j+4,&& \ldots,& (j-1)j+j\},\\
&\{j+1,&&2j+2,&& 3j+3,&& 4j+4,&& \ldots,& j\},\\
&\{2j+1,&&3j+2,&& 4j+3,&& 5j+4,&& \ldots,& j+j\},\\
\vdots \\
&\{(j-1)j+1,&&2,&& j+3,&& 2j+4,&& \ldots,& (j-2)j+j\},
\end{align*}
it would be observed that each of the numbers $u,\; u=1,\,\ldots,\,j$,
occurs  as a summand in one and only one member of each set 
and the same is true for each of the numbers $uj,\; u=1,\,\ldots,\,j-1$. It follows that the sum of the members in each set is the same, the common sum being $j(j^2+1)/2$. Further, each set contains $j$ integers and it is readily seen that the integers in all the $j$ sets put together are just a permutation of the first $j^2$ consecutive positive integers. Thus the lemma is proved.
\end{proof}

\begin{lem} Any solution of the multigrade chain \eqref{basicchnnot1} yields a solution of the multigrade chain 
\begin{equation}
b_{11},\,b_{21},\ldots,\,b_{t1} \stackrel{k+1}{=} b_{12},\,b_{22},\ldots,\,b_{t2} \\
\stackrel{k+1}{=}  \ldots \stackrel{k+1}{=} b_{1j},\,b_{2j},\ldots,\,b_{tj} \label{basicchnnot1b}
\end{equation}
where $t=js$.
\end{lem}
\begin{proof} Let $h_1,\,h_2,$  $\ldots,\,h_j$ be an arbitrary  set of $j$ distinct integers. We take the integers $b_{u1},\;u=1,\,2,\,\,\ldots,\,t$, as follows:
\begin{equation}
\begin{aligned}
&a_{11}+h_1,\;a_{21}+h_1,\,\ldots,\,a_{s1}+h_1, \\
&a_{12}+h_2,\;a_{22}+h_2,\,\ldots,\,a_{s2}+h_2,\\
&\vdots \\
&a_{1j}+h_j,\;a_{2j}+h_j,\,\ldots,\,a_{sj}+h_j.
\end{aligned} 
\label{setb1}
\end{equation}
For any given integer $v$ where $2 \leq v \leq j$, we replace  $h_1,\,h_2,$  $\ldots,\,h_j$  in the set of integers \eqref{setb1} by $h_v,\,h_{v+1},\,\ldots,\,\,h_{v+j-1}$ respectively where we take $h_m=h_{m-j}$  when $m > j$, and the resulting integers are taken to be  the integers $b_{uv},\;u=1,\,2,\,\,\ldots,\,t$. We will now show that, with these values of $b_{uv}$, the relations \eqref{basicchnnot1b} are satisfied.

The proof is by the multinomial theorem. In view of the relations \eqref{basicchnnot1},  it is readily seen that the relations \eqref{basicchnnot1b} are true for exponents $1,\,2,\,\ldots,\,k$. Further,  when we consider the relation \eqref{basicchnnot1b}  for the exponent $k+1$, on expanding the terms of the first set, that is, $b_{u1}^{k+1},\;u=1,\,\ldots,\,t$, and adding only the terms involving $h_1^r,\,h_2^r,\,\ldots,\,h_j^r$ where $1 \leq r \leq k+1$,  we get
\[
\begin{aligned}
&\sum_{u=1}^s \binom{k+1}{r}a_{u1}^{k+1-r}h_1^r+\sum_{u=1}^s \binom{k+1}{r}a_{u2}^{k+1-r}h_2^r+\cdots+\sum_{u=1}^s \binom{k+1}{r}a_{uj}^{k+1-r}h_j^r\\
& \quad =(h_1^r+h_2^r+\cdots+h_j^r)\sum_{u=1}^s \binom{k+1}{r}a_{u1}^{k+1-r}.
\end{aligned}
\]
It is now easy to see that  the terms involving $h_i^r,\;i=1,\,2,\,\ldots,\,j$, where $1 \leq r \leq k+1$, add up to the same common sum in  each set. Further, the terms independent of $h_i$ add up to $\sum_{u=1}^s\sum_{v=1}^ja_{uv}^{k+1}$ in  each set.  It is thus seen that the relations \eqref{basicchnnot1b} are also true for the exponent $k+1$. This proves the lemma. 
\end{proof}

\section{Multigrade chains consisting only of the first $N$ consecutive positive integers}

In Section 3.1 we give three theorems which show that there exist infinitely many integers $N=js$ such that the consecutive positive integers $1,\,2,\,\ldots,\,N$ can be separated into $j$ sets, each set consisting of $s$ integers, such that the $j$ sets provide a solution of \eqref{basicchnnot1} for a certain value of $k$. In Section 3.2 we give some numerical examples of such multigrade chains.

\subsection{}
\begin{thm} If $N=2j^k$ where $j \geq 2$ and $k \geq 1$, the first $N$ consecutive positive integers can be separated into $j$ sets in at least $\{(j-1)!\}^{k-1}$ ways, each set consisting of $2j^{k-1}$ integers, such that the $j$ sets   provide a solution of the multigrade chain \eqref{basicchnnot1}.
\end{thm}
\begin{proof} The proof is by induction. It follows from Lemma 2 that the result is true when $k=1$. 

We now assume that the result is true when $k=n$, that is,   
we assume that there exist integers $a_{uv},\;u=1,\ldots,\,s,\;v=1,\ldots,\,j,$ where $s=2j^{n-1}$ such that
\begin{equation}
a_{11},\,a_{21},\ldots,\,a_{s1} \stackrel{n}{=} a_{12},\,a_{22},\ldots,\,a_{s2} \\
\stackrel{n}{=}  \ldots \stackrel{n}{=} a_{1j},\,a_{2j},\ldots,\,a_{sj}, \label{mchn1}
\end{equation}
and the integers $a_{ij}$ are a permutation of the first $2j^n$ positive integers.

On applying Lemma 1 with $M=j,\;K=-j$ to the relations \eqref{mchn1}, we get the multigrade chain,
\begin{equation}
b_{11},\,b_{21},\ldots,\,b_{s1} \stackrel{n}{=} b_{12},\,b_{22},\ldots,\,b_{s2} \\
\stackrel{n}{=}  \ldots \stackrel{n}{=} b_{1j},\,b_{2j},\ldots,\,b_{sj}, \label{mchn2}
\end{equation}
where  the integers $b_{ij}$ are a permutation of  the integers $0,\,j,\,2j,\,\ldots,\,2j^{n+1}-j$.

We now apply Lemma 5 to the relations \eqref{mchn2} taking the integers $h_1,\,h_2,$ $\ldots,\,h_j$, as the integers $1,\,2,\,\ldots,\,j$, and we get the multigrade chain,
\begin{equation}
c_{11},\,c_{21},\ldots,\,c_{t1} \stackrel{n+1}{=} c_{12},\,c_{22},\ldots,\,c_{t2} \\
\stackrel{n+1}{=}  \ldots \stackrel{n+1}{=} c_{1j},\,c_{2j},\ldots,\,c_{tj}, \label{mchn3}
\end{equation}
where  $t=2j^n$ and the integers  $c_{uv},\;u=1,\ldots,\,t,\;v=1,\ldots,\,j$, are obtained by adding each of the integers  $1,\,2,\,\ldots,\,j$ to each of the integers $0,\,j,\,2j,\,\ldots,\,2j^{n+1}-j$. It follows that the integers $c_{uv}$ are the consecutive integers $1,\,2,\,\ldots,\,2j^{n+1}$. Thus, the first $2j^{n+1}$ positive integers have been separated into $j$ sets, each set consisting of $2j^n$ integers, such that the $j$ sets provide a solution of the multigrade chain \eqref{basicchnnot1} with $k=n+1$. 

In fact, we may take  the integers $h_1,\,h_2,\,\ldots,\,h_j$ to be any permutation of the integers $1,\,2,\,\ldots,\,j$, and we still get a  multigrade chain of type \eqref{mchn3} consisting of the consecutive integers $1,\,2,\,\ldots,\,2j^{n+1}$. For getting distinct multigrade chains of type \eqref{mchn3}, we may keep $h_1=1$ as fixed while permuting the remaining $j-1$ integers in $(j-1)!$ ways. Thus, starting from the multigrade chain \eqref{mchn1},  we get $(j-1)!$ distinct multigrade chains \eqref{mchn3} consisting of the consecutive integers $1,\,2,\,\ldots,\,2j^{n+1}$. The theorem  now follows by induction.  
\end{proof}

\begin{thm} If $N=2mj^k$, the first $N$ consecutive positive integers can be separated into $j$ sets in at least $\{(j-1)!\}^{k-1}$ ways, each set consisting of $2mj^{k-1}$ integers, such that the $j$ sets  provide a solution of the multigrade chain \eqref{basicchnnot1}.
\end{thm}
\begin{proof} By Lemma 3, the result is true for $k=1$. The remaining proof is similar to that of Theorem 6 and is accordingly omitted. 
\end{proof}

\begin{thm} If $N=j^{k+1}$ where $j \geq 2$ and $k \geq 1$, the first $N$ consecutive positive integers can be separated into $j$ sets in at least $\{(j-1)!\}^{k-1}$ ways, each set consisting of $j^k$ integers, such that the $j$ sets  provide a solution of the multigrade chain \eqref{basicchnnot1}.
\end{thm}
\begin{proof}  By Lemma 4, the result is true for $k=1$. As in the case of Theorem 7, the remaining proof is similar to the   proof of Theorem 6 and is omitted. This gives yet another proof of Prouhet's result.
\end{proof}

\subsection{}
We now give a few numerical examples. Since $18=2.3^2$, in view of Theorem 6, the consecutive integers 1,\,2,\,\ldots,\,18 can be separated into 3 sets -- each set consisting of 6 integers -- to yield two multigrade chains valid for exponents and 1 and 2. These two multigrade chains are as follows:
\begin{equation}
 1, 5, 9, 12, 14, 16 \stackrel{2}{=} 2, 6, 7, 10, 15, 17          \stackrel{2}{=}3, 4, 8, 11, 13, 18, \label{chn18ex1}
\end{equation}
and
\begin{equation} 1, 6, 8, 11, 15, 16              \stackrel{2}{=}  3, 5, 7, 10, 14, 18             \stackrel{2}{=}2, 4, 9, 12, 13, 17. \label{chn18ex2}
\end{equation}											
We note that the smallest exponent in the canonical prime factorization of 18 is 1, and hence the method described by Roberts \cite{Ro2, Ro3} does not generate the above multigrade chains.

As a second example, in view of Theorem 8, the first 27 consecutive positive integers can be separated into three sets -- each set having 9 integers -- to yield two multigrade chains. These two multigrade chains are as follows:
\begin{equation}
\begin{aligned}
1,\, 6,\, 8,\, 11,\, 13,\, 18,\, 21,\, 23,\, 25&\stackrel{2}{=}2,\, 4,\, 9,\, 12,\, 14,\, 16,\, 19,\, 24,\, 26\\
&\stackrel{2}{=}3,\, 5,\, 7,\, 10,\, 15,\, 17,\, 20,\, 22,\, 27.
\end{aligned}
\label{chn27ex2}
\end{equation}
and
\begin{equation}
\begin{aligned}
1,\, 5,\, 9,\, 12,\, 13,\, 17,\, 20,\, 24,\, 25&\stackrel{2}{=} 2,\, 6,\, 7,\, 10,\, 14,\, 18,\, 21,\, 22,\, 26\\
 &\stackrel{2}{=}3,\, 4,\, 8,\, 11,\, 15,\, 16,\, 19,\, 23,\, 27. 
\end{aligned}
\label{chn27ex3}
\end{equation}

It is interesting to observe that both of the above multigrade chains are distinct from the one given by Prouhet. In fact,\, there is a fourth multigrade chain comprising of the first 27 positive integers. It is as follows:
\begin{equation}
\begin{aligned}
1,\, 5,\, 9,\, 11,\, 15,\, 16,\, 21,\, 22,\, 26 &\stackrel{2}{=} 2,\, 6,\, 7,\, 12,\, 13,\, 17,\, 19,\, 23,\, 27\\
&\stackrel{2}{=}3,\, 4,\, 8,\, 10,\, 14,\, 18,\, 20,\, 24,\, 25.
 \end{aligned}
\label{chn27ex4}
\end{equation}

\section{An open problem}

It follows from the Theorems 6, 7 and 8 that, for any given positive integers $k \geq 1$ and $j \geq 2$, there exist infinitely many integers $N$ such that the first $N$ consecutive positive integers can be separated into $j$ sets that provide a solution of the multigrade chain \eqref{basicchnnot1}. Accordingly for $k \geq 1$ and $j \geq 2$, we  define $N(k,\,j)$ to be the least positive integer $N$ with this property. An immediate consequence of Theorem 6 is that $ N(k,\,j) \leq 2j^k$. It would be of interest to determine the integer $N(k,\,j)$.

It is readily proved that $N(1,\,j)=2j$, $N(2,\,2)=8$ and $N(2,\,3)=18$. Thus, in these cases $N(k,\,j)=2j^k$. In fact, it appears that  $N(k,\,j)=2j^k$ for arbitrary positive integers $k \geq 1$ and $j \geq 2$ but this remains to be proved.

\begin{center}
\Large
Acknowledgments
\end{center}
 
I wish to  thank the Harish-Chandra Research Institute,  Prayagraj for providing me with all necessary facilities that have helped me to pursue my research work in mathematics.

\medskip

\noindent Postal Address: Ajai Choudhry, 
\newline \hspace{1.05 in}
13/4 A Clay Square,
\newline \hspace{1.05 in} Lucknow - 226001, INDIA.
\newline \noindent  E-mail: ajaic203@yahoo.com
\end{document}